\documentclass[10pt]{article}

\usepackage{amsbsy,amssymb,amsmath,amsthm}
\usepackage{caption,color,graphicx}
\usepackage[a4paper,text={6.5in,10in},centering]{geometry}
\usepackage[colorlinks=true,hypertexnames=false,linkcolor=blue,citecolor=blue]{hyperref}
\usepackage[numbers,comma,square,sort&compress]{natbib}
\usepackage{authblk}

\setcaptionmargin{0.25in}

\newtheorem{Lemma}{Lemma}[section]

\newtheorem{Definition}{Definition}[section]

\newtheorem{Theorem}{Theorem}[section]

\title{Circle diffeomorphisms forced by expanding circle maps}

\author[1,2]{Ale Jan Homburg} 
\affil[1]{KdV Institute for Mathematics, University of Amsterdam, Science park 904, 1098 XH Amsterdam, Netherlands}
\affil[2]{Department of Mathematics, VU University Amsterdam, De Boelelaan 1081, 1081 HV Amsterdam, Netherlands}

\begin{document}
\maketitle

\begin{abstract}
We discuss dynamics of skew product maps defined by circle diffeomorphisms forced by expanding circle maps.
We construct an open class of such systems that are robust topologically mixing and for which almost all points in the same fiber converge
under iteration. This property follows from the construction of 
an invariant attracting  graph in the natural extension, a skew product of
circle diffeomorphisms forced by a solenoid homeomorphism.\\ \\
MSC 37C05, 37D30, 37C70, 37E10
\end{abstract}

\section{Introduction}\label{s:i}

We will treat the dynamics of a class of circle diffeomorphisms that are forced by expanding circle maps.
We start with a numerical experiment on the skew product map
\begin{align}
(y,x) &\mapsto (3 y, x + \frac{1}{8}\sin (2 \pi x) + y) \mod 1
\end{align}
on the torus $\mathbb{T}^2 = (\mathbb{R}/\mathbb{Z})^2$,
the results of which are presented in Figure~\ref{f:attractor}.
Note that this map is given by a circle diffeomorphism $x\mapsto x + \frac{1}{8}\sin (2 \pi x) + y \mod 1$ in the fiber forced by an expanding circle map $y \mapsto 3y \mod 1$ in the base.
The left panel of 
Figure~\ref{f:attractor} shows ten thousand points of an orbit, appearing to lie dense in the torus. The right panel shows a time series of the second coordinate of twenty different orbits, for equidistant distributed initial points in the same fiber (i.e. 
with identical first coordinate). There appears to be a fast contraction
inside the fiber.

\begin{figure}[ht]
\centerline{
\hbox{
\includegraphics[height=6cm,width=6cm]{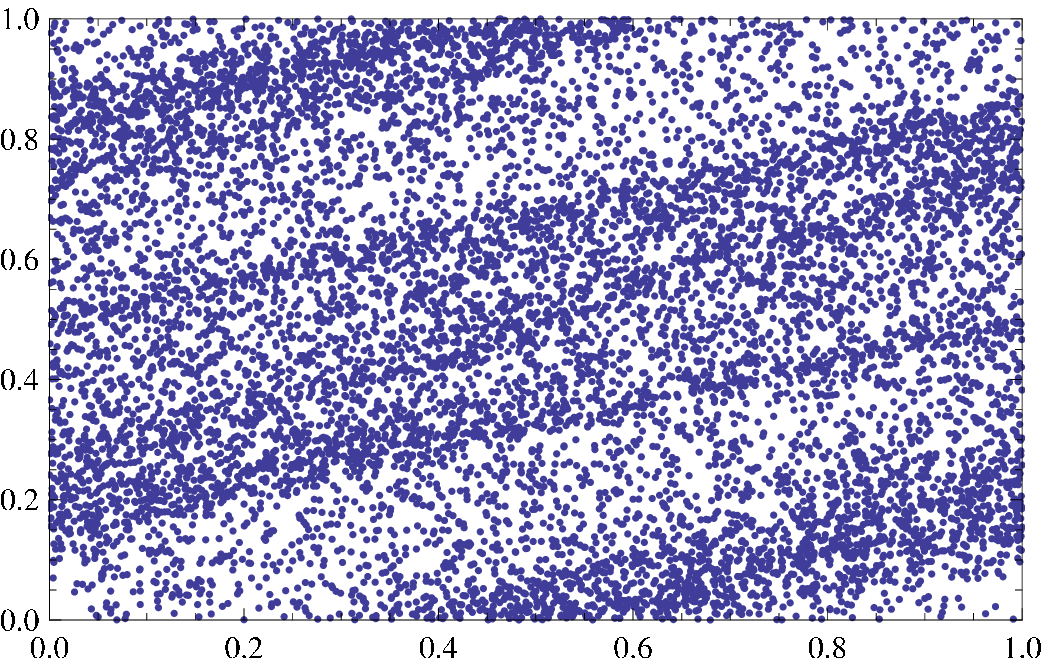}
\hspace{1in}
\includegraphics[height=6cm,width=7cm]{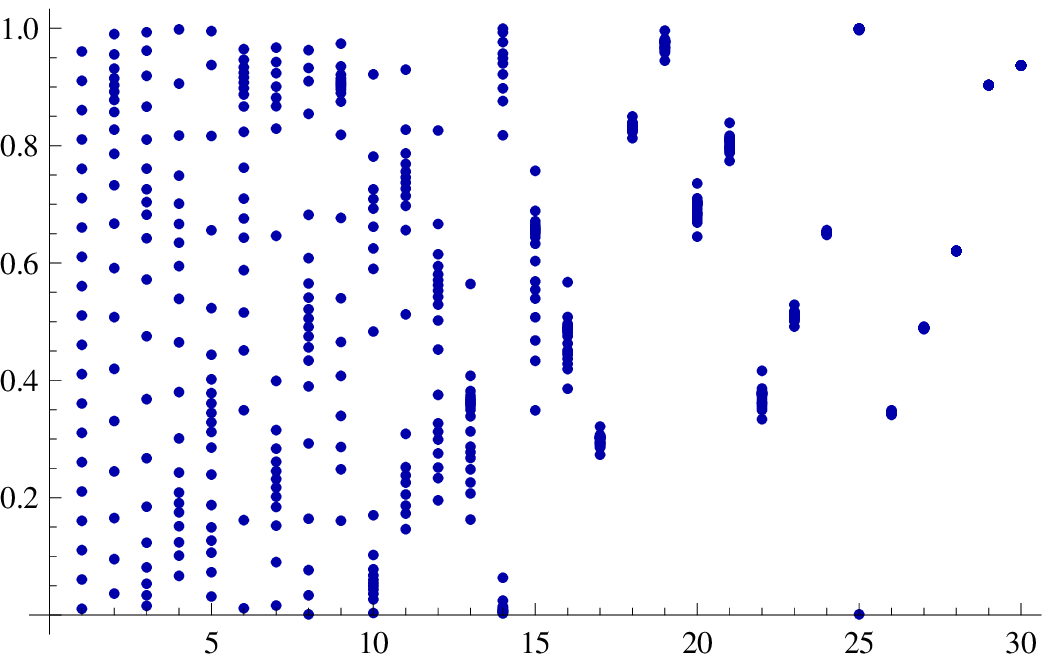}
}}
\caption{Numerical experiments on $(y,x) \mapsto (3 y, x + \frac{1}{8}\sin (2 \pi x) + y)  \mod 1$. The left frame shows ten thousand points of an orbit. The right frame shows time series for the $x$-coordinate starting from twenty different initial conditions with identical $y$-coordinates.
\label{f:attractor}}
\end{figure}

Numerical experiments as described above may be explained by the following 
result. Endow the space of smooth skew product systems $(y,x)\mapsto F(y,x) = (g(y) , f_y(x))$ on $\mathbb{T}^2$, considered as a subset of smooth endomorphisms, with the $C^k$ topology, $k\ge 2$.
A smooth endomorphism $g$ on the circle is called expanding if $|g'| >1$.
Recall that $F$ is topologically mixing if for each nonempty open $U,V \subset \mathbb{T}^2$,
$F^n(U)$ intersects $V$ for all large enough positive integers $n$; this implies the existence of
dense positive orbits.

\begin{Theorem}\label{t:main}
There is an open class of forced circle diffeomorphisms $(y,x)\mapsto F(y,x) = (g(y) , f_y(x))$,
forced by expanding circle maps $y \mapsto g(y)$, 
with the following properties:
\begin{enumerate}
 \item each map $F$ is topologically mixing,
 \item there is a subset $\Lambda \subset \mathbb{T}^2$ of full Lebesgue measure,
so that for any $(y,x_1), (y,x_2) \in \Lambda$, 
\begin{align}\label{e:synch}
\lim_{n\to \infty}  | F^n(y,x_1) - F^n(y,x_2) | &= 0.
\end{align}
\end{enumerate}
\end{Theorem}

To prove the convergence property in Theorem~\ref{t:main}, we apply the natural extension of the endomorphism $F$ to a
homeomorphism on the product of a solenoid and a circle (this construction is described in Section~\ref{s:ne}).
This homeomorphism is likewise a skew product map formed by 
a circle diffeomorphism forced by a solenoid map.
It is shown to admit an attracting invariant graph (Theorem~\ref{t:omega} below), 
from which the result follows.

In the physics literature  a convergence phenomenon as in \eqref{e:synch} 
falls under the study of synchronization, see \citep{MR1913567} for a review.
Forced circle maps appear in various contexts where some sort of convergence of orbits features.
We give pointers to the literature for these different contexts and discuss the relation to our result.

\begin{enumerate}

\item
Quasiperiodically forced circle diffeomorphisms; the circle diffeomorphisms $f_y(x)$ are forced by
$g(y) = y + \alpha \mod 1$ with $\alpha$ irrational.
A large body of work is available in this area of research, related to the existence of strange non-chaotic attractors, 
see \citep{MR2504850} and references therein. 
The notable difference with the context here is that the forcing consist of ergodic, but not mixing, dynamics.

 \item 
Randomly perturbed circle diffeomorphisms, including
iterated function systems \citep{ant84,MR2358052,1086.37026}
and circle diffeomorphisms with absolutely continuous noise \citep{lej87,MR2425065}.
Such systems allow a formulation as a skew product system.
The mentioned references give precise classifications of dynamics in 
the fibers both for iterated function systems and circle diffeomorphisms
with absolutely continuous i.i.d. noise.
For an iterated function system consisting of $m$ circle diffeomorphisms $f_1, \ldots,f_m$
this yields circle diffeomorphisms forced by a shift on $m$ symbols:
consider $\Sigma = \{ 1,\ldots,m \}^\mathbb{N}$
endowed with the product topology and, for $\boldsymbol{\omega} = (\omega_0,\omega_1,\ldots) \in \Sigma$,
the left shift  $\sigma \boldsymbol{\omega} = (\omega_1,\omega_2, \ldots)$.
The skew product system $F$ acting on $\Sigma \times \mathbb{T}$ is then given by 
\begin{align*}
 F(\boldsymbol{\omega},x) &= (\sigma \boldsymbol{\omega}, f_{\omega_0} (x)).
\end{align*}
The left shift is a topologically mixing map.
The dependence of the circle diffeomorphisms on $\boldsymbol{\omega}$ is of a restricted form:
they depend only on $\omega_0$ and not on $\omega_i, i>0$ (in \cite{1002.37017} the term 
``step skew product'' is used).
Our result can be seen as an extension where 
this restriction is removed and also as an extension to more 
general topologically mixing base dynamics.

\item
(Volume preserving) skew products over hyperbolic torus automorphisms  \citep{MR1738057,MR1838747,MR1972184}. One may think 
of small perturbations from $F: \mathbb{T}^2 \times \mathbb{T} \to \mathbb{T}^2 \times \mathbb{T}$,
\begin{align*}
 F(y,x) &=(A y , x ),
\end{align*}
where $A$ is a hyperbolic torus automorphism.
This research relates to the phenomenon of 
stable ergodicity. It also relates to work on partially hyperbolic systems
with mostly contracting central directions
\citep{bondiavia05,MR2068774}.
The above references contain results on 
delta measures in fibers, which go in the direction of
the convergence result in Theorem~\ref{t:main}. Reference \cite{hom} combines the approaches
of \citep{MR1738057,MR1838747} and this paper,
and contains a result akin to Theorem~\ref{t:main}. 
As reviewed in Section~\ref{s:ph}, one may embed the solenoid from the natural extension of the expanding circle map
as a hyperbolic attractor for a smooth diffeomorphism on a manifold,
so that the natural extension of the skew product system is partially hyperbolic on $\mathcal{S}\times \mathbb{T}$.

\end{enumerate}

Finally, skew product systems of circle diffeomorphisms 
over horseshoes and solenoids are also treated in \citep{1002.37017} 
with a different emphasis, proving the robust occurrence of dense sets of hyperbolic periodic orbits 
with different index (attracting or repelling in the fiber).
In \citep{MR2132437}, and continued in \citep{MR2545014,bondiagor}, the existence of
ergodic measures with zero Lyapounov exponent is investigated
in related contexts of skew product systems and partially hyperbolic systems.

\section{Natural extensions}\label{s:ne}

In this section we collect mostly known facts on extensions of 
skew product torus endomorphisms to skew product homeomorphisms.
These facts are used in the arguments in the following sections.

Consider a smooth expanding endomorphism 
$g : \mathbb{T} \to \mathbb{T}$ on the circle  $\mathbb{T} = \mathbb{R}/\mathbb{Z}$.
We note that $g$ possesses an absolutely continuous invariant measure
$\nu^+$,  equivalent to Lebesgue measure, see e.g. \citep[Section~III.1]{MR889254}.
In fact, $g$ is topologically conjugate to a linear expanding circle map for which
Lebesgue measure is invariant \citep{shub}.
The measure $\nu^+$ has density that is bounded and bounded away from zero.
We will consider skew product systems of circle diffeomorphisms $x \mapsto f_y(x)$
forced by the expanding circle map $y \mapsto g(y)$.
Write
\begin{align}
  F(y,x) &= (g(y) , f_y(x))
\end{align}
for the skew product map on the torus $\mathbb{T}^2$.
For iterates of $F$, we denote
\begin{align*}
 F^n(y,x) &= (g^n(y) , f_{g^{n-1}(y)}\circ\cdots\circ f_y(x))=(g^n(y) , f^n_y(x)).  
\end{align*}

The inverse limit construction \citep{MR0348794} extends $g$ to a homeomorphism on the solenoid, i.e. the space 
\begin{align*}
\mathcal{S} &=
\{ (\ldots,y_{-1,},y_0) \in \mathbb{T}^{-\mathbb{N}} \mid y_{-i} = g(y_{-i-1}) \}
\end{align*} 
endowed with the product topology.
We will also write $g$ for the extended map, where the context makes clear 
whether $g$ acts on $\mathbb{T}$ or $\mathcal{S}$. 
So, for $\mathbf{y} = (\ldots,y_{-1},y_0)$,  
\begin{align*}
g(\mathbf{y}) &= (\ldots,y_{-1},y_0,g(y_0)).
\end{align*} 
The induced skew product map on $\mathcal{S}\times \mathbb{T}$ will likewise be denoted by $F$, and we write 
\begin{align*}
F (\mathbf{y},x) &= (g(\mathbf{y}) , f_{\mathbf{y}} (x)).
\end{align*}
The inverse map is given by
\begin{align*}
F^{-1}(\mathbf{y},x) = (\ldots, y_{-2}, y_{-1}, (f_{y_{-1}})^{-1} (x)).
\end{align*}

On $\mathbb{T}$ and $\mathcal{S}$ we use Borel $\sigma$-algebras $\mathcal{F}^+$ and
$\mathcal{F}$ respectively. 
Define the projection $\psi: \mathcal{S} \to \mathbb{T}$; $\psi(\mathbf{y}) = y_0$.
Then with $\mathcal{G} = \psi^{-1} (\mathcal{F}^+)$ we have $g^n \mathcal{G} \uparrow \mathcal{F}$
and $g: \mathcal{S} \to \mathcal{S}$ is a natural extension of $g :\mathbb{T} \to \mathbb{T}$
\citep[Appendix~A]{arn98}.  
The solenoid $\mathcal{S}$ has an invariant measure $\nu$ inherited from the 
invariant measure $\nu^+$ for $g$ on the circle;
\[
\nu (\{ y_{-r} \in I_r, \ldots, y_0 \in I_0 \}) = \nu^+ (g^{-r}(I_0) \cap g^{-r+1} (I_{1}) \cap \cdots \cap I_r).
\]
We will write $\lambda$ for Lebesgue measure on $\mathbb{T}$.
We also write $|I|$ for the length of an interval $I \subset \mathbb{T}$.

Let $\mu^+$ be an invariant measure for $F: \mathbb{T}^2 \to \mathbb{T}^2$ with marginal $\nu^+$; existence is guaranteed by \citep[Lemma~2.3]{cra90}. 
Write $\mu^+_{y}$ for the disintegrations of $\mu^+$.
Occasionally we also write $\mu^+_\mathbf{y}$ with the understanding that $\mu^+_\mathbf{y}$
depends only on the coordinate $y_0$ in $\mathbf{y} = (\ldots,y_{-1},y_0)$.
Invariance of $\mu^+$ means
\begin{align*}
 \int_{g^{-1}(A)} f_y \mu_y^+ d \nu^+(y) &= \int_{g^{-1}(A)} \mu^+_{g (y)} d\nu^+(y) = \int_A \mu^+_y d\nu^+(y)
\end{align*}
for $A \in \mathcal{F}^+$ (the second equality by invariance of $\nu^+$ under $g$), see \citep{cra90}, \citep[Theorem~1.4.5]{arn98}.

The following lemma originating from \citep{cra90} relates invariant measures for the skew product system with one and two sided time.

\begin{Lemma}\label{l:furstenberg}
Given the invariant measure $\mu^+$ for $F$ acting on $\mathbb{T}^2$,
with marginal $\nu^+$ on $\mathbb{T}$,
there is an invariant measure $\mu$ for $F$ acting on $\mathcal{S}\times \mathbb{T}$,
with marginal $\nu$ on $\mathcal{S}$.
For $\nu$-almost all $\mathbf{y}  = (\ldots, y_{-1}, y_0)\in \mathcal{S}$,
the limit 
\begin{equation}\label{e:to+}
\mu_{\mathbf{y}} = \lim_{n\to\infty} f^n_{y_{-n}} \mu^+_{y_{-n}}
\end{equation}
gives its disintegrations. 
\end{Lemma}

\begin{proof}
The lemma is implied by \cite[Theorem 1.7.2]{arn98}. We include the line of reasoning.
To avoid confusion we write $\mathcal{B}$ (and not again $\mathcal{F}^+$) 
for the Borel $\sigma$-algebra on the circle of $x$-coordinates.
For fixed $B \in \mathcal{B}$, and for $\mathbf{y} = (\ldots, y_{-n},\ldots,y_0) \in \mathcal{S}$,
define  
\begin{align*}
\nu_{\mathbf{y}}^n (B) &=  f^n_{g^{-n}(\mathbf{y})} \mu^+_{y_{-n}} (B)
\end{align*}
as the push-forward by $f^n_{g^{-n}(\mathbf{y})}$ of $\mu^+_{y_{-n}}$, evaluated in $B$. Recall $\mathcal{G} = \psi^{-1} \mathcal{F}^+$, with $\mathcal{G}_n = g^{n} \mathcal{G}$
we have $\mathcal{G}_n \uparrow \mathcal{F}$ as $n \to \infty$.
One computes that $\mathbb{E}(\nu^n_{\mathbf{y}} (B)  | \mathcal{G}_m) = \nu^m_{\mathbf{y}}(B)$,
i.e. $y\mapsto \nu_{\mathbf{y}}^n (B)$ is a martingale with respect to the filtration $\mathcal{G}_n$.
As this holds for all fixed $B$, 
$\mu_{\mathbf{y}} (B) = \lim_{n\to\infty} f^n_{y_{-n}} \mu^+_{y_{-n}} (B)$
defines a probability measure for $\nu$-almost all $\mathbf{y}$.
\end{proof}

Vice versa, given an invariant measure $\mu$ for $F$ on $\mathcal{S}\times \mathbb{T}$,
\begin{equation}\label{e:from+}
 \mu^+_y = {\mathbb E}( \mu | \mathcal{F}^+)_y
\end{equation}
is an invariant measure for $F$ on $\mathbb{T}^2$ \cite[Theorem 1.7.2]{arn98}.
Moreover, the correspondence maps ergodic measures to ergodic measures in either direction
\cite[Section~3]{cra90}.


We will also need to study iterates of the inverse map  $F^{-1}$ on $\mathcal{S} \times \mathbb{T}$.
Noting that this interchanges stable and unstable directions, 
one obtains a convergence result similar to Lemma~\ref{l:furstenberg}. 
To state it, it is convenient to think of $g$ as acting on $[0,1]$; one can identify $0$ with $1$ to obtain the expanding circle map. 
The inverse limit construction extends $g$ to a map, also denoted by $g$, on 
\begin{align*}
\mathcal{I} &= \{ (\ldots, y_{-1}, y_0 ) \in [0,1]^{-\mathbb{N}} \;|\; y_{-i} = g (y_{-i-1}) \}. 
\end{align*}
We may think of $g$ as acting on $\Sigma \times [0,1]$
for a Cantor set $\Sigma = \{0,\ldots,m-1\}^\mathbb{N}$. 
The solenoid $\mathcal{S}$ is then given as a quotient 
\begin{align}\label{e:SI}
\mathcal{S} &= \Sigma\times [0,1]/ \sim,
\end{align}
identifying points $(\boldsymbol{\omega},0)$ and $(\boldsymbol{\nu},1)$ in $\Sigma \times \{0,1\}$ 
for which $g^{-1} (\boldsymbol{\omega},0) = g^{-1}(\boldsymbol{\nu},1)$.
We write $\mathbf{y} = (\boldsymbol{\omega},y_0) \in \Sigma \times \mathbb{T}$.

Consider the projection $\psi : \Sigma\times [0,1] \to \Sigma$, 
$\psi (\boldsymbol{\omega},y_0) = \boldsymbol{\omega}$.
The Borel $\sigma$-algebra on $\Sigma\times [0,1]$ is 
$\mathcal{F} = \mathcal{F}^- \otimes \mathcal{F}^+$.
The inverse map $g^{-1}$ on $\Sigma\times [0,1]$ induces an expanding map on $\Sigma$ with an invariant measure $\nu^-$ (with $\nu$ the invariant measure for $g$ on $\Sigma\times [0,1]$).
Write $\mathcal{G} = \psi^{-1} \mathcal{F}^-$. 
The measure $\nu^-$ is computable from $\nu^+$: for a cylinder 
$C = C_{\nu_1\ldots\nu_k} = \{\boldsymbol{\omega}\;|\; \omega_i = \nu_i \text{ for } i=1,\ldots,k\}$,
it satisfies $\nu^-(C) = \nu (F^{-k} (C \times [0,1]) ) = \nu^+ (J)$ 
with $F^{-k} (C \times [0,1]) = \Sigma\times J$.
Now $g^{-n} \mathcal{G} \uparrow \mathcal{F}$ and $g^{-1}: \mathcal{I} \to \mathcal{I}$ is the natural extension of
$g^{-1}: [0,1]\to [0,1]$.
By a continuously differentiable coordinate change, the strong unstable lamination $\mathcal{F}^{uu}$
is affine; $\mathcal{F}^{uu} = \{ (\boldsymbol{\omega},y,x) \;|\; \boldsymbol{\omega},x\; \mathrm{ constant} \}$. This makes $F^{-1}$ like $F$ up to interchanging
strong stable and  strong unstable directions.
In the resulting coordinates, write 
$F^{-1} (\mathbf{y},x) = (g^{-1} (  \mathbf{y}) , k^{-1}_{\mathbf{y}} (x))$ (where $k^{-1} _{\mathbf{y}} (x)$ 
depends only on $\boldsymbol{\omega}$ and $x$).
Suppose $\zeta^{-}$ is an invariant measure for $F^{-1}$ on $\Sigma\times [0,1] \times \mathbb{T}$ 
with $\sigma$-algebra $\mathcal{G} \otimes \mathcal{B}$ and with marginal $\nu^-$.
We write $\zeta^-_{\boldsymbol{\omega}}$, $\boldsymbol{\omega} \in \Sigma$, or also
$\zeta^-_{\mathbf{y}}$,
for its disintegrations.

\begin{Lemma}\label{l:furstenberg-}
Given the invariant measure $\zeta^-$ for $F^{-1}$ acting on $\Sigma\times \mathbb{T}$, 
with marginal $\nu^-$ on $\Sigma$, there is an invariant measure $\zeta$ for $F^{-1}$ acting on $\Sigma\times [0,1]\times \mathbb{T}$ 
with marginal $\nu$ on $\Sigma\times [0,1]$.
For $\nu$-almost all $\mathbf{y}\in \mathcal{S}$, the limit
\begin{equation}\label{e:to-}
\zeta_{\mathbf{y}} = \lim_{n\to\infty} k^{-n}_{g^{n} (\mathbf{y})} \zeta^-_{g^{n} (\mathbf{y})}
\end{equation}
gives it disintegrations.
\end{Lemma}

\begin{proof}
As for Lemma~\ref{l:furstenberg} one can apply \cite[Theorem 1.7.2]{arn98} to prove the lemma.
\end{proof}

\section{Partial hyperbolicity}\label{s:ph}

See e.g. \cite[Section~17.1]{kathas97} for the standard construction of the solenoid as an attractor for a diffeomorphism on $(-1,1)^2\times\mathbb{T}$.
Likewise the solenoid can appear as an attractor for a diffeomorphism on 
$(-1,1)^d\times\mathbb{T}$, $d\ge 2$.
Under an assumption
\begin{align}\label{e:parthyp}
 m = \max_{y,x} f'_y (x) <& \min_{y} g'(y) = M,
 \end{align}
one may embed the solenoid as a hyperbolic attractor,
so that the class of skew product systems is partially hyperbolic \citep{bondiavia05}  on $\mathcal{S}\times \mathbb{T}$.  
This results in a partially hyperbolic splitting in 
one-dimensional strong unstable directions, one-dimensional
center directions (the fibers) and the remaining $d$-dimensional strong stable directions. 
Write $N = (-1,1)^d \times \mathbb{T}^2$ for the (open neighborhood in the) manifold
that contains $\mathcal{S}\times \mathbb{T}$ as hyperbolic attractor;
the map $F$ on $\mathcal{S}\times \mathbb{T}$ is extended to a diffeomorphism $F$ on $N$.

Write $W^{ss}(\mathbf{y},x)$ for the strong stable manifold of
$(\mathbf{y},x)$ and 
$W^{uu}(\mathbf{y},x)$ for the strong unstable manifold of
$(\mathbf{y},x)$.
The strong stable and strong unstable manifolds form laminations $\mathcal{F}^{ss}$ and
$\mathcal{F}^{uu}$. 

\begin{Lemma}
Assuming \eqref{e:parthyp}, there exists 
an embedding of $\mathcal{S}$ as a hyperbolic attractor for a smooth diffeomorphism on a manifold,
so that the class of skew product systems is partially hyperbolic on $\mathcal{S}\times \mathbb{T}$.
For $\max_{y,x} \{f'_y (x),1/f'_y(x)\}$ sufficiently close to 1 and $M>2$,
such an embedding exists for which  $\mathcal{F}^{ss}$ and 
$\mathcal{F}^{uu}$ are continuously differentiable laminations.
\end{Lemma}

\begin{proof}
In the strong stable directions, taking the dimension $d$ sufficiently large
(depending on the degree of the expanding circle map $g$), 
distances can be assumed to be contracted by a factor 
close to $\frac{1}{2}$. 

Observe that, forced by the form of the map $F$, the local strong stable manifold $W^{ss}_\mathrm{loc}(\mathbf{y},a_0)$ 
for any $\mathbf{y} \in \psi^{-1} (y_0)$
equals $\psi^{-1} (y_0) \times \{a_0\}$.
The strong stable lamination is
%
%
therefore continuously differentiable.
If $f'_y (x)$ is near $1$ for all $x,y$, 
then with $M>2$ (the expanding map $g$ has to be of degree three or higher)
spectral gap conditions are satisfied that imply that
the strong unstable lamination is continuously differentiable. 
This is checked by going through the construction of the strong unstable lamination by graph transform techniques \citep{MR0501173}, as we will indicate.

One obtains the strong unstable lamination by integrating the
line field formed by the strong unstable directions. 
Write $T N = N \times E^{uu} \times E^{ss,c}$ so that the strong unstable directions at a point
$x\in \mathcal{S}\times \mathbb{T}$ are given as the graph of a linear map
in $\mathcal{L}(E^{uu},E^{ss,c})$. 
The strong unstable directions are then given by the graph of a section
$\mathcal{S}\times \mathbb{T} \mapsto \mathcal{L}(E^{uu},E^{ss,c})$ that is invariant
under the induced diffeomorphism $\hat{F} : \mathcal{S}\times \mathbb{T} \times \mathcal{L}(E^{uu},E^{ss,c}) \to \mathcal{S}\times \mathbb{T} \times \mathcal{L}(E^{uu},E^{ss,c})$;
\begin{align}
  \hat{F}(\mathbf{y},x,\alpha) &= (F(\mathbf{y},x) , \beta), \qquad  \mathrm{graph }\,\beta = DF (\mathbf{y},x) \mathrm{graph }\,\alpha.
\end{align}  
It is possible to construct strong unstable directions on $N$ that extend those
on $\mathcal{S}\times \mathbb{T}$ by choosing a lamination on a fundamental 
domain in its basin of attraction and iterating under the graph transform
\cite[Appendix~1]{pt}.
This produces a graph $V^{uu}$  of a section 
$N \mapsto \mathcal{L}(E^{uu},E^{ss,c})$ that is invariant under $\hat{F}$.
 
If $\lambda^{ss}$ is the strongest rate of contraction, i.e. for some $C>0$ and
$i \in \mathbb{N}$,
\begin{align*} 
|DF^{i} (n) v| &\ge C (\lambda^{ss})^{i} |v|,
\end{align*}   
for each $n\in N$, $v \in T_xN$,
then such a graph $V^{uu}$ is normally hyperbolic for $m / M < \lambda^{ss}$.
Indeed, the contraction of $\hat{F}$ along the fibers $\mathcal{L}(E^{uu},E^{ss,c})$ is  estimated by
\begin{align}
   D\hat{F}^i(n,\alpha) (0,w) &\le C (m/M)^i |w|,
\end{align}
compare \citep[Appendix~1]{pt}.
Normal hyperbolicity holds for $\lambda^{ss}$ near 
$\frac{1}{2}$, $m$ near $1$ and $M >2$.
Normal hyperbolicity implies that $V^{uu}$ is continuously differentiable and
this in turn implies that the strong unstable lamination is continuously differentiable 
\citep{MR1432307}. 
\end{proof}

\section{Robust transitivity}\label{s:rt}

We record that 
\[
F_{i,j} (y,x) = (iy , x  + jy) \mod 1, 
\]
with $i>1, j$ integers, is not topologically transitive;
it leaves all circles parallel to $jx = (i-1)y  \mod 1$ invariant.
Note that $F_{i,j}$ induces a homeomorphism on $\mathcal{S} \times \mathbb{T}$;
$F_{i,j} (\mathbf{y},x) = (i \mathbf{y} , x + i y_0)$, with inverse
$F^{-1}_{i,j} (\mathbf{y},x) = (\ldots ,y_{-2},y_{-1},x - j y_{-1})$.

The following result provides a class of robust topologically mixing skew product maps.
We use ad hoc arguments, relying on the skew product structure with topologically mixing base dynamics, 
to prove it, 
but the arguments bear a resemblance to the technique of blenders introduced in \citep{MR1381990}.

\begin{Theorem}\label{t:robust_trans}
There exist arbitrarily small smooth perturbations $F$, $F(y,x) = (g(y),f_y(x))$, 
of $F_{i,0}$, $i>1$, 
that are robustly topologically mixing skew product maps
(considered on either $\mathbb{T}^2$ or $\mathcal{S}\times\mathbb{T}$).

Moreover, 
\begin{enumerate}
 \item 
there are $k\in\mathbb{N}$, $\hat{y}\in\mathbb{T}$,
with
$g^k(\hat{y}) = \hat{y}$ and $f^k_{\hat{y}}$ possessing a unique hyperbolic attracting and hyperbolic repelling fixed point,
\item
for any $(\mathbf{y},x) \in \mathcal{S}\times \mathbb{T}$,
the strong stable and strong unstable manifolds 
$W^{ss} (\mathbf{y},x)$, $W^{uu} (\mathbf{y},x)$
are dense in $\mathcal{S}\times \mathbb{T}$. 
\end{enumerate}
\end{Theorem}

\begin{proof}
Consider $\Sigma_n^+ = \{0,\ldots,n\}^\mathbb{N}$ endowed with the product topology
and let $\sigma: \Sigma_n^+ \to \Sigma_n^+$ be the left shift.
The base map $g$ (or some iterate thereof) 
admits invariant Cantor sets on which the dynamics 
is topologically conjugate to $\sigma: \Sigma_n^+ \mapsto \Sigma_n^+$.
This observation and the following lemma imply the existence of 
robust topologically mixing maps $F$ acting on $\mathbb{T}^2$ as stated in the theorem. 

\begin{Lemma}\label{l:ifs}
There exists a
skew product map
\[
H(\boldsymbol{\omega},x) = (\sigma \boldsymbol{\omega}, h_{\boldsymbol{\omega}} (x)) 
\]
on $\Sigma_n^+\times \mathbb{T}$, $n\ge 4$, that is robustly topologically mixing
under continuous perturbations of 
$\boldsymbol{\omega} \mapsto h_{\boldsymbol{\omega}}$ in the $C^1$ topology.
\end{Lemma}

\begin{proof}
Following \citep{ifs}, take circle diffeomorphisms
$h_0$, $h_1$, $h_2$ so that
\begin{enumerate}
 \item $h_i$ has a unique hyperbolic attracting fixed point $p_i$ and a
unique hyperbolic repelling fixed point $q_i$, $i=0,1,2$; the fixed points are mutually disjoint,
 \item $p_0$, $p_1$ are close to each other and $h_0, h_1$ are affine on
 $[p_0,p_{1}]$,
 \item $p_2\in (p_0,p_1)$,
 \item $\frac{1}{2} < (h_0)' (p_0),(h_1)' (p_{1})  < 1$.
\end{enumerate}
The iterated function system generated by $h_0,h_1,h_2,h_3 = h_2^{-1}$ is robustly minimal under $C^1$ small perturbations of $h_0,\ldots,h_3$.
We give the main steps in the reasoning, referring to \citep{ifs} for details.
Consider the iterated function system generated by $h_0,h_1$. For a compact subset 
$S \subset \mathbb{T}$, write $\mathcal{L} (S) = h_0(S)\cup h_1(S)$.
Let $E_{in} \subset [p_0,p_1] \subset E_{out}$ be intervals close to $[p_0,p_1]$ on which
$h_0,h_1$ are affine.
Then 
\begin{equation}\label{e:ifsinclusion}
E_{in} \subset \mathcal{L}(E_{in}) \subset [p_0,p_1] \subset \mathcal{L}(E_{out}) \subset E_{out}
\end{equation}
and $\mathcal{L}^i (E_{in}), \mathcal{L}^i(R_{out})$ converge to $[p_0,p_1]$ 
in the Hausdorff topology as $i\to \infty$.
Since $h_0$ and $h_1$ are contractions, this shows that the iterated function system
generated by $h_0,h_1$ is minimal on $[p_0,p_1]$.
From the properties of $h_2,h_3$ it is easily concluded that the iterated function system
generated by $h_0,h_1,h_2,h_3$ is minimal on $\mathbb{T}$.


The skew product system 
$H (\boldsymbol{\omega},x) = (\sigma \boldsymbol{\omega},h_{\omega_0}(x))$ is topologically mixing.
Indeed, write $\Sigma_2^+ = \{0,1\}^\mathbb{N} \subset \Sigma_n^+$ and take an open set 
$U\subset \Sigma_2^+ \times [p_0,p_{1}]$. A high iterate $H^n (U)$ contains a strip $\Sigma^+_2 \times J$  
in $\Sigma_2^+ \times [p_0,p_{1}]$. 
Now $H^{n+1}(U)$ maps $\Sigma^+_2 \times J$ to two strips with total
width larger than $c |J|$ with $c = (h_0)'(p_0) + (h_1)'(p_1) > 1$.
Further iterates $H^{n+k} (U)$ contain $2^k$ strips
of increasing total width so that for some $k>0$,  $H^{n+k} (U)$ 
lies dense in $\Sigma_2^+ \times I$ for any $I \subset [p_0,p_{1}]$.
Iterates of $\Sigma_n^+ \times [p_0,p_1]$ under $H$ 
lie dense in  
$\Sigma_n^+ \times \mathbb{T}$
since the repelling fixed point $q_2$ of $h_2$ lies inside $[p_0,p_1]$.
This shows that $H$ is topologically mixing.

This reasoning also applies to small perturbations of $H$, where also
the fiber maps may depend on all of $\boldsymbol{\omega}$ instead of just
$\omega_0$. 
We note the following changes in the reasoning.
The inclusions \eqref{e:ifsinclusion} get replaced by
\begin{equation*}
\Sigma_2^+ \times E_{in} \subset H (\Sigma_2^+ \times E_{in}), \qquad
  H(\Sigma_2^+ \times E_{out}) \subset \Sigma_2^+ \times E_{out}
\end{equation*}
The map $H$ acting on $\Sigma_2^+ \times E_{out}$ acts by contractions in the fibers
$\boldsymbol{\omega} \times E_{out}$.
A high iterate  $H^n (U)$ may not contain a product $\Sigma_2^+ \times J$
but contains a strip of some width $\varepsilon$ lying between the graphs of two maps
$\Sigma_2^+ \to \mathbb{T}$. Again $H^{n+1}(U)$ contains two strips of total width
exceeding $c \varepsilon$ for some $c>1$, and 
$H^{n+k} (U)$ contain $2^k$ strips
of increasing total width. We conclude that 
there is an interval $[\tilde{p}_0,\tilde{p}_1]$ near $[p_0,p_1]$ so that 
for some $k>0$,  
$H^{n+k} (U)$ 
lies dense in 
$\Sigma_2^+ \times I$ for any $I \subset [\tilde{p}_0,\tilde{p}_{1}]$.

If $\boldsymbol{\omega}$ starts with a sequence of $i$ symbols $2$, then
$h_{\sigma^i \boldsymbol{\omega}}\circ \cdots\circ  h_{\boldsymbol{\omega}}$
maps an interval $I\subset \mathbb{T}$ that contains $q_2$ to an interval 
with length approaching 1 as $i\to\infty$.
Also, any point in $\mathbb{T}$ can be mapped into
$[\tilde{p}_0,\tilde{p}_1]$ by an iterate that involves $\boldsymbol{\omega}$ 
with a long sequence of symbols $3$.

The lemma follows.
\end{proof}

As a consequence, $F$ acting on $\mathbb{T}^2$ is topologically mixing. 
Indeed, take an open set $U$ in $\Sigma_n^+ \times \mathbb{T}$.
The construction in Lemma~\ref{l:ifs} gives that 
$\cup_{n\in \mathbb{N}} F^n(U)$ is open and dense in  
$\Sigma_n^+ \times \mathbb{T}$.
Now take open sets $U,V \subset \mathbb{T}^2$. As $g$ is expanding,
some iterate of $U$ under $F$ intersects  $\Sigma_n^+ \times \mathbb{T}$.
Again as $g$ is expanding, a higher iterate will intersect $V$, establishing
topological mixing of $F: \mathbb{T}^2 \to \mathbb{T}^2$.  



It easily follows that also $F$ acting on $\mathcal{S}\times \mathbb{T}$ 
is topologically mixing.
Just note that $g^n(\ldots,y_{-1},y_0) = (\ldots,g^{n-1}(y_0),g^n(y_0))$ and
an open set in the product topology is 
of the form $\mathcal{S} \cap (\ldots U_{-2} \times U_{-1} \times U_0)$ with
$U_{-i} \subset \mathbb{T}$ open, and proper only for finitely many
values of $i$.

The first property in the list in the theorem follows 
as part of the above construction. 
The construction also implies that strong unstable manifolds are dense
in $\mathcal{S}\times \mathbb{T}$.
To see this, consider the periodic points $P$ and $Q$ for $F:\mathcal{S}\times \mathbb{T} \to \mathcal{S}\times \mathbb{T}$,
where $P$ corresponds to $p_0$ in the proof of the lemma and $Q$ corresponds to $q_2$.
The two-dimensional unstable manifold of $Q$ lies dense in
$\mathcal{S}\times \mathbb{T}$ since unstable manifolds for $g$ lie dense in $\mathcal{S}$.
Note that the stable manifold of $Q$ contains
points arbitrarily close to $P$.
We claim that $W^{u}(Q) \subset \overline{W^{uu}(P)}$
(compare \citep[Lemma~1.9]{MR1381990}):
take a point $x \in W^{u}(Q)$ and a neighborhood $V$ of it, iterate backwards and note that $F^{-m} (V)$ intersects $W^{uu}(P)$.
Thus $W^{uu}(P)$, and therefore each strong unstable manifold, lies dense in $\mathcal{S}\times \mathbb{T}$. 


Finally use these arguments
for inverse diffeomorphisms, making further small perturbations, to see that there are skew product maps for which also strong stable manifolds lie dense in $\mathcal{S}\times \mathbb{T}$.
\end{proof}


\begin{Definition}
The skew-product map  $F$ is called strongly contractive if
for all $\varepsilon >0$, there are $\hat{y}\in \mathbb{T}$, an interval $V \subset {\mathbb S}^1$, $n \in {\mathbb N}$, so that $|V| > 1-\varepsilon$ and 
$ | f^n_{\hat{y}}(I) | < \varepsilon$. 
\end{Definition}

The following lemma, that provides a robust condition for $F$ being strongly contractive, is immediate. 

\begin{Lemma}\label{l:attr_fixed}
Suppose there exist $k\in\mathbb{N}$, $\hat{y}\in\mathbb{T}$,
with
$g^k(\hat{y}) = \hat{y}$ and $f^k_{\hat{y}}$ possessing a unique hyperbolic attracting and hyperbolic repelling fixed point.
Then $F$ is strongly contractive.
\end{Lemma}

\section{Attracting invariant graphs}\label{s:aig}

Contraction of positive orbits starting in the same fiber, is explained by the following result. 
Theorem~\ref{t:main} follows from it.
The arguments that establish random fixed points
in iterated function systems, see \citep[Proposition 5.7]{MR2358052} (compare also \citep[Section~2.3]{navas}) 
and \citep{MR2425065}, 
are based on pushing forward a stationary measure by the circle diffeomorphisms
and identifying limit measures. 
Although there is no stationary measure in our context, our proof of Theorem~\ref{t:omega} is inspired
by this approach.
Different approaches using the theory of nonuniform hyperbolic systems to provide invariant delta-measures
in forced circle diffeomorphisms, are followed in 
\citep{lej87,cra90,MR1738057}.
Such approaches do not determine the number of points in each fiber and 
would therefore not allow to explain Theorem~\ref{t:main}.

\begin{Theorem}\label{t:omega}
Let $F: \mathbb{T}^2 \to \mathbb{T}^2$ 
be robust topologically mixing as in Theorem~\ref{t:robust_trans},
so that $F$ is also strongly contractive. 
Then $F$ acting on $\mathcal{S} \times \mathbb{T}^2$
admits an invariant graph $\{ (\mathbf{y},\omega^+ (\mathbf{y})) \mid \mathbf{y} \in \mathcal{S} \}$
for a measurable function $\omega^+$, that attracts 
the positive orbits of $\nu \times \lambda$-almost all initial points.
\end{Theorem}

\begin{proof}
The main steps in the proof are the following.
We show that for $\nu$-almost all $\mathbf{y}$, 
the push-forwards $f^n_{g^{-n} (\mathbf{y})} \lambda$ of Lebesgue measure $\lambda$ contain delta measures in the fiber over $\mathbf{y}$ 
as accumulation points in the weak star topology.
Invoking Lemma~\ref{l:furstenberg}, we establish that $f^n_{g^{-n} (\mathbf{y})} \lambda$ in fact converges to a delta measure, 
thus proving the existence of an invariant graph for $F$ acting on $\mathcal{S}\times\mathbb{T}$.
For the attraction properties, we must likewise consider $F^{-1}$ and construct an invariant graph for $F^{-1}$.

We start with a lemma.

\begin{Lemma}\label{l:aecontr}
Given $\varepsilon >0$, for $\nu$-almost all $\mathbf{y} \in \mathcal{S}$ there are an interval $I\subset \mathbb{T}$ with $|I|>1-\varepsilon$ and $n \in \mathbb{N}$, 
so that $|f^n_{g^{-n}(\mathbf{y})}(I)| < \varepsilon$.
\end{Lemma}

\begin{proof}
The lemma will be a consequence of a construction in which we 
provide $\delta >0$, $L \in {\mathbb N}$ so that the following holds.
Given an interval $J \subset \mathbb{T}$ and $n\in\mathbb{N}$, we construct for each $y \in J$, an interval $I \subset \mathbb{T}$ with
$|I|>1-\varepsilon$, an open subset $J' \subset \psi^{-1}(J) \subset \mathcal{S}$ with $\nu(J')/\nu^+(J) = \nu(J')/\nu(\psi^{-1}(J))>\delta$ and a positive integer $l\le L$, 
so that for $\mathbf{y} \in J'$, $|f^{n + l}_{g^{-n-l}(\mathbf{y})}( I)| < \varepsilon$. 

Fix $\varepsilon>0$.

\noindent \emph{Step 1.}
There are intervals $K \subset \mathbb{T}$, $V \subset \mathbb{T}$
with $|V| > 1-\varepsilon$ and $N \in \mathbb{N}$, 
so that $|f^N_y (V)| < \varepsilon$ for $y \in K$. 
We make this more explicit.
By Theorem~\ref{t:robust_trans}, $f^k_{\hat{y}}$ has a hyperbolic attracting fixed point $p$. 

By taking suitable smooth coordinates near the forward orbit of $(\hat{y},p)$, 
we may assume that the local unstable manifold $W^{u}_\text{loc}(\hat{y},p)$ of $(\hat{y},p)$ is contained in $\mathbb{T} \times \{p\}$.

The interval $K$ can be taken a fundamental domain (i.e. an interval for which $g^k$ maps one boundary point to the other), so that $g^{i}(K)$
stays close to the forward orbit of $\hat{y}$ for $0\le i \le N$.
We let $N$ be a multiple of $k$, so that $g^N(K)$ is close to $\hat{y}$.
By replacing $K$ with $g^{-ki} (K)$ for some $i>0$,
and replacing $N$ by $N + ki$, we can decrease the size of the image 
$f^N_y (V)$, while keeping $g^N(K)$ fixed.
Write $U_y = f^N_y (V)$ and also
$V_z = f^N_y (V)$ with $z = g^N(y)$.

\noindent \emph{Step 2.}
Take $q > 1/\varepsilon$. Take $\zeta_0 = y_{-n}$ in $J$. 
Write $\hat{\mathbf{y}}$ for the periodic point of $g: \mathcal{S}\to\mathcal{S}$ in $\psi^{-1}(\hat{y})$.
Iterates of an interval $N \subset W^{uu}_\mathrm{loc}(\hat{\mathbf{y}},p)$ 
lie dense in $\mathcal{S}\times \mathbb{T}$, by Theorem~\ref{t:robust_trans} and expansion properties of $g$. 
One can therefore take a point $\boldsymbol{\zeta} = (\ldots,\zeta_0) \in \mathcal{S}$ so that 
\begin{enumerate}
 \item there are positive integers $M_1 < \ldots < M_q$ so that 
$\zeta_{-M_i} \in g^N(K)$, $1 \le i \le q$,
\item with $a_i, b_i$ given by 
$(\zeta_{-M_i},b_i) \in  W^{uu}_\mathrm{loc} (\hat{\mathbf{y}},p)$
and $(\zeta_{-M_1} , a_i) = F^{M_i-M_1} (\zeta_{-M_i},b_i)$,
the points $a_i$, $1 \le i \le q$, are disjoint.
%
\end{enumerate}

\noindent \emph{Step 3.}
Take neighborhoods $L_i \subset g^N(K)$ of $\zeta_{-M_i}$ so that $f^{M_i}_{z} (V_z)$, $z \in L_i$,  are disjoint for different $i$.
Consider $\cap_{1 \le i \le q} g^{M_i} (L_i)$. Since finitely many such intervals (for varying $y_0$, $J$) cover the
circle $\mathbb{T}$, the numbers $N, M_i$ are bounded (that is, depend only on $\varepsilon$ and the dynamical system $F$).

\noindent \emph{Step 4.}
Let $L = \cup_{1\le i\le q} L_i \subset g^N(K)$.
Let $O = \cup_{1\le i\le q} g^{M_i}(L_i)$.
For $y \in J \cap O$, there is $j$, $1 \le j \le q$, with
$f^n_y (V_z)$ with $z = g^{-M_j} (y) \cap L_j$ that has length smaller than $\varepsilon$ (since there are $q > 1/\varepsilon$ such disjoint intervals).
This defines a set of $y$ values for which one of 
$f^{N+M_j+n}_y (U_y)$ is small.

For given $\varepsilon >0$, there is a bound $\delta>0$ with $|\cup_{1\le i \le q} L_i|/|J| > \delta$. 
A similar bound holds with $\nu^+$ replacing the length of intervals, since $\nu^+$
has density that is bounded and bounded away from zero. 
This ends the construction.
Now define 
\[\Delta_N = \{ \mathbf{y}\in \mathcal{S} \mid \text{for each interval } I \text{ with } |I| > 1 -\varepsilon \text{ and each } i \le N,\; |f^i_{g^{-i}(\mathbf{y})} (I)| > \varepsilon\}.\]
The above construction yields the estimate $\nu(\Delta_{tl}) \le (1-\delta)^t$.
Thus $\nu(\Delta) = 0$, where
\[\Delta  = \{ \mathbf{y}\in \mathcal{S} \mid \text{for each interval } I \text{ with } |I| > 1 -\varepsilon \text{ and each } i,\; |f^i_{g^{-i}(\mathbf{y})} (I)| > \varepsilon\}.\tag*{\qedhere}\]
\end{proof}

Lemma~\ref{l:aecontr} implies that the push-forwards
$f_{g^{-n}(\mathbf{y})}^n \lambda$ contain a delta-measure $\delta_{\omega^+ (\mathbf{y})}$,
concentrated at $\omega^+ (\mathbf{y})$, as accumulation point.
This yields an invariant graph
$\{ (\mathbf{y},\omega^+ (\mathbf{y})) \mid \mathbf{y} \in \mathcal{S} \}$ 
for $F: \mathcal{S}\times \mathbb{T} \to \mathcal{S}\times \mathbb{T}$.
Let $\mu^+_{y_0}$ be obtained from 
$\delta_{\omega^+ (\mathbf{y})}$, $\mathbf{y} =  (\ldots,y_0)$,
as in \eqref{e:from+}.
The following lemma will be applied to find that $f_{g^{-n}(\mathbf{y})}^n \lambda$
and $f_{g^{-n}(\mathbf{y})}^n \mu^+_{y_{-n}}$ 
converge to the delta measure $\delta_{\omega^+ (\mathbf{y})}$ for $\nu$-almost all $\mathbf{y} \in \mathcal{S}$.
We refer to \citep{tsujii,tsujii-acta} for general results on 
invariant measures for partially hyperbolic endomorphisms.
Recall that a measure is diffuse if it has no atoms.

\begin{Lemma}\label{l:suppmu+}
For each $y_0 \in \mathbb{T}$, 
 \begin{align}\label{e:suppmu+}
  \mathrm{supp}\; \mu^+_{y_0} = \mathbb{T}.
 \end{align}
Moreover, $\mu^+_{y_0}$ is diffuse and depends continuously 
on $y_0$ in the weak star topology. 
\end{Lemma}

\begin{proof}
We use an estimate $m = \max_{y,x} \{ f'_y (x) , 1/f'_y(x)\} < \min_{y} g'(y) = M$,
which is implicit in Theorem~\ref{t:robust_trans}.
Consider $\mathbf{z}$ close to $\mathbf{y}$; i.e. $z_i$ close to $y_i$ for all $i \in -\mathbb{N}$.  
The branch of $g$ defined near $y_{-i-1}$ for which $g(y_{-i-1}) = g_{-i}$ has an inverse;
in the following we write $g^{-1}$ for it with the understanding that we consider
orbits near $\mathbf{y}$. 
Consider $f^{n}_{g^{-n}(y)} (x) =  f_{g^{-1}(y)} \circ \ldots \circ f_{g^{-n} (y)} (x)$ and compute
\begin{align*}
 \frac{\partial}{\partial y}  f^{n}_{g^{-n}(y)} (x) 
&= \sum_{i=1}^n 
\left( f^{i-1}_{g^{-i+1}(\mathbf{y})}  \right)' ( f^{n-i+1}_{g^{-n}(\mathbf{y})}(x) )
\frac{\partial}{\partial y}  f_{g^{-i}(\mathbf{y})} ( f^{n-i}_{g^{-n}(\mathbf{y})} (x)) 
\left( g^{-i} \right)' (\mathbf{y}),
\end{align*}
which is uniformly bounded by \eqref{e:parthyp}.
Likewise
\begin{align*}
 \frac{\partial}{\partial y}  f^{n-l}_{g^{-n}(y)} (x) M^{l} &= \mathcal{O} (1).
\end{align*}

Now, for a subsequence $n_i \to \infty$, $f^{n_i}_{y_{-n_i}} \lambda$ converges in the weak star topology 
to a delta measure $\delta_{\omega^+ (\mathbf{y})}$.
Take $l$ so that $g^{-l} (\mathbb{T})$ is an interval of length 
$\mathcal{O} (\varepsilon^s)$ for a positive $s$. 
Recall that $f^n_{g^{-n}(\mathbf{y})}$ maps an interval $V$ of length
$1 - \varepsilon$ to an interval $I$ of length $\varepsilon$.  
Then
$f^{-l}_\mathbf{y} (I)$ is an interval of length $\varepsilon^t$ for some $t >0$.  
These estimates imply that for $\mathbf{z}$ near $\mathbf{y}$, $f^{n_i}_{g^{-n_i}(z)}\lambda$ 
converges to a delta measure $\delta_{\omega^+(\mathbf{z})}$ depending continuously on $\mathbf{z}$. 
The graph transform construction of the strong unstable lamination in fact shows 
that $(\mathbf{z},\omega^+(\mathbf{z}))$ is in 
$W^{uu} ( \mathbf{y},\omega^+(\mathbf{y}))$. See also \citep{MR721733} and \citep[Chapter~11]{bondiavia05}.

Consider local center stable manifolds
$W^{ss,c}(y) = \{((\ldots,y_{-1},y_0),x) \in \mathcal{S}\times \mathbb{T}, \; y_0=y  \}$
in $\mathcal{S}\times\mathbb{T}$.
The invariant measure $\mu$ has disintegrations $\mu_y$ along
$W^{ss,c} (y)$, $y \in\mathbb{T}$. 
If $\pi^{ss}$ denotes the projection onto the fiber $\mathbb{T}$,
\begin{align*}
\pi^{ss} ((\ldots,y_{-1},y_0),x) &= (y_0,x),
\end{align*}
then 
\begin{align}\label{e:mu+mu}
\mu^+_{y_0} &= \pi^{ss}\mu_{y_0}.
\end{align}
We claim that the disintegrations $\mu_{y_0}$ are $u$-invariant, 
meaning that the disintegrations
$\mu_{y_0}$ are invariant under the holonomy along strong unstable leaves.

Consider $F$ acting on $\Sigma \times \mathbb{T}^2$ (compare Section~\ref{s:ne}) and 
take coordinates in which the strong unstable lamination is affine.
Take a product measure $m = \nu_2\times\nu$.
A C\'esaro accumulation point of push-forwards $F^n m$ 
is a Gibbs $u$-measure \citep{MR721733,MR2366230}, which is unique \cite{MR1749677}.
The C\'esaro accumulation point is a product measure and hence 
$u$-invariant (see also  \cite[Remark~4.1]{avivia10}).

Equation \eqref{e:suppmu+} follows by \eqref{e:from+} since strong unstable manifolds 
are dense in $\mathcal{S}\times \mathbb{T}$.
If an open set has positive measure, also the image under $F$ has positive measure.
Since the measure $\mu$ is invariant and the strong unstable lamination is minimal, 
with \eqref{e:mu+mu} this yields
\eqref{e:suppmu+}.
Continuous dependence of $\mu^+_{y_0}$ on $y_0$ is 
implied by \eqref{e:mu+mu} and $u$-invariance of $\mu_{y_0}$.
Since $\mu_\mathbf{y}$ is ergodic, the measure $\mu^+_{y_0}$ is ergodic. 
In view of \eqref{e:suppmu+} it is therefore diffuse.  
 \end{proof}

\begin{Lemma}\label{l:todelta}
For $\nu$-almost all $\mathbf{y} \in \mathcal{S}$,
\[
 \lim_{n \to \infty}  f_{g^{-n}(\mathbf{y})}^n \mu^+_{g^{-n}(\mathbf{y})} = 
 \lim_{n \to \infty} f_{g^{-n}(\mathbf{y})}^n \lambda =
\delta_{\omega^+ (\mathbf{y})}
\]
for a delta measure $\delta_{\omega^+ (\mathbf{y})}$.
\end{Lemma}

\begin{proof}
Recall that $f_{g^{-n}(\mathbf{y})}^n \lambda$
has a delta measure $\delta_{\omega^+ (\mathbf{y})}$ as accumulation point.
Further, $f_{g^{-n}(\mathbf{y})}^n \mu^+_{g^{-n}(\mathbf{y})}$
converges by Lemma~\ref{l:furstenberg}. 
By Lemma~\ref{l:suppmu+},  $f_{g^{-n}(\mathbf{y})}^n \lambda$ and
$f_{g^{-n}(\mathbf{y})}^n \mu^+_{g^{-n}(\mathbf{y})}$ converge to $\delta_{\omega^+ (\mathbf{y})}$.
\end{proof}

We have constructed an invariant graph
$\{ (\mathbf{y},\omega^+(\mathbf{y})\}$ for $F: \mathcal{S}\times \mathbb{T} \to \mathcal{S}\times \mathbb{T}$.
To prove its attraction property, we need to consider iterates from time zero to time $n>0$.
Invertibility of the maps in the fibers implies that if $f^n_{\mathbf{y}} \lambda$ is close to a delta measure,
then also $f^{-n}_{g^n(\mathbf{y})} \lambda$ is close to a delta measure.
So we can also consider iterates from
time $n>0$ to time 0, for which we consider the inverse skew product map $F^{-1}$.
One can largely follow the previous reasoning to construct
an invariant graph $\{ (\mathbf{y},\omega^-(\mathbf{y})\}$ for $F^{-1}: \mathcal{S}\times \mathbb{T} \to \mathcal{S}\times \mathbb{T}$.

We give the lemma's that correspond to Lemma's~\ref{l:suppmu+} and \ref{l:todelta}.
Recall the last part of Section~\ref{s:ne} on ergodic properties of $F^{-1}$.

\begin{Lemma}\label{l:suppzeta-}
For each $\boldsymbol{\omega} \in \Sigma$, 
 \begin{align}\label{e:suppzeta1-}
  \mathrm{supp}\; \zeta^-_{\boldsymbol{\omega}} = \mathbb{T}.
 \end{align}
Moreover,  $\zeta^-_{\boldsymbol{\omega}}$ is diffuse and depends continuously on $\boldsymbol{\omega}$ in the weak star topology.
\end{Lemma}

\begin{Lemma}\label{l:todelta-}
For $\nu$-almost all $\mathbf{y} \in \mathcal{S}$,
\[
 \lim_{n \to \infty}  k_{g^{n}(\mathbf{y})}^{-n} \zeta^-_{g^{n}(\mathbf{y})} = 
 \lim_{n \to \infty} f_{g^{n}(\mathbf{y})}^{-n} \lambda =
\delta_{\omega^- (\mathbf{y})}
\]
for a delta measure $\delta_{\omega^- (\mathbf{y})}$.
\end{Lemma}

Lemma~\ref{l:todelta-} implies that the graph of 
$\omega^+$, whose existence is given by Lemma~\ref{l:todelta}, is attracting.
It attracts all points lying outside the graph of $\omega^-$.
This is true even if $\omega^+ = \omega^-$, but in fact 
$\omega^+(\mathbf{y}) \neq \omega^-(\mathbf{y})$ for $\nu$-almost
all $\mathbf{y}$.
This can be seen by writing $\mathcal{S} = \Sigma \times I/\sim$ as in \eqref{e:SI},
so that in $\mathbf{y} = \boldsymbol{\omega} \times y$ the ``past'' $\boldsymbol{\omega}$ 
and the ``future'' $y$ are independent. 
The resulting positions $\omega^+(\mathbf{y})$ and 
$\omega^-(\mathbf{y})$ depend on past and future only, respectively,  
and vary according to Lemma~\ref{l:todelta} and Lemma~\ref{l:todelta-}.

This finishes the proof of Theorem~\ref{t:omega}. 
\end{proof}



\def\cprime{$'$}

\end{document}